\documentclass[a4paper,reqno, 11pt]{amsart}   	
\usepackage[DIV=12, oneside]{typearea}			
\usepackage[utf8]{inputenc}
\usepackage[T1]{fontenc}
\usepackage[english]{babel}
\usepackage[centertags]{amsmath}
\usepackage{amstext,amssymb,amsopn,amsthm}
\usepackage{mathrsfs}
\usepackage{dsfont}
\usepackage{bbm} 
\usepackage{thmtools}
\usepackage{mathtools}
\usepackage{graphicx}
\usepackage[titletoc,title]{appendix}
\usepackage{etexcmds}

\usepackage[backgroundcolor=white, bordercolor=blue,linecolor=blue]{todonotes}

\usepackage{multirow}

\usepackage{tikz}
\usepackage{pgfplots}
\pgfplotsset{compat=newest}
\usetikzlibrary{arrows.meta, decorations.pathreplacing, positioning}
\usetikzlibrary{calc, patterns,angles,quotes, tikzmark, fit, shapes.geometric}
\tikzset{%
	ebo unit/.store in=\ebounit,
	ebo corners/.style={rounded corners=#1\ebounit},
}
\usepgfplotslibrary{fillbetween}
\definecolor{c595959}{RGB}{89,89,89}
\definecolor{c5a5a5a}{RGB}{90,90,90}

\usepackage[colorlinks=true, linkcolor=black, urlcolor={black},
pdfborder={0 0 0}, citecolor=black]{hyperref}

\usepackage{enumitem}
\setlist[enumerate]{itemsep=0mm}

\addto\extrasenglish{}
\addto\extrasenglish{}
\addto\extrasenglish{}

\theoremstyle{plain}
\declaretheorem[title=Theorem, parent=section]{theorem}

\declaretheorem[title=Corollary,sibling=theorem]{corollary}

\theoremstyle{definition}

\declaretheorem[title=Remark,sibling=theorem]{remark}
\declaretheorem[title=Remark, numbered=no]{remark*}

\declaretheorem[title=Assumption, numbered=no]{assumption*}

\numberwithin{equation}{section}

\newcommand{\N}{\mathds{N}}
\newcommand{\R}{\mathds{R}}

\def\hmath$#1${\texorpdfstring{{\rmfamily\textit{#1}}}{#1}}

\newcommand{\eps}{\varepsilon}

\newcommand{\loc}{\mathrm{loc}}

\newcommand{\BIGOP}[1]
{
	\mathop{\mathchoice%
		{\raise-0.22em\hbox{\huge $#1$}}%
		{\raise-0.05em\hbox{\Large $#1$}}{\hbox{\large $#1$}}{#1}}}

\def\XXint#1#2#3{{\setbox0=\hbox{$#1{#2#3}{\int}$}
		\vcenter{\hbox{$#2#3$}}\kern-.5\wd0}}

\newcommand{\BIGboxplus}{\mathop{\mathchoice%
		{\raise-0.35em\hbox{\huge $\boxplus$}}%
		{\raise-0.15em\hbox{\Large $\boxplus$}}{\hbox{\large $\boxplus$}}{\boxplus}}}

\DeclareMathOperator{\pv}{p.v.}

\renewcommand{\d}{\textnormal{d}}

\usepackage{esint}
 
\makeatletter
\def\namedlabel#1#2{\begingroup
	#2%
	\def\@currentlabel{#2}%
	\phantomsection\label{#1}\endgroup
}
\makeatother

\makeatletter
\newcommand{\hathat}[1]{%
	\begingroup%
	\let\macc@kerna\z@%
	\let\macc@kernb\z@%
	\let\macc@nucleus\@empty%
	\hat{\mathchoice%
		{\raisebox{.2ex}{\vphantom{\ensuremath{\displaystyle #1}}}}%
		{\raisebox{.2ex}{\vphantom{\ensuremath{\textstyle #1}}}}%
		{\raisebox{.16ex}{\vphantom{\ensuremath{\scriptstyle #1}}}}%
		{\raisebox{.14ex}{\vphantom{\ensuremath{\scriptscriptstyle #1}}}}%
		\smash{\hat{#1}}}%
	\endgroup%
}
\makeatother

\begin{document}
	\allowdisplaybreaks
	\title{A strong-type unique continuation principle for the fractional $p$-Laplacian} 
	
	\author{Florian Grube}

	\address{Fakult{\"a}t f{\"u}r Mathematik, Universit{\"a}t Bielefeld, Postfach 10 01 31, 33501 Bielefeld, Germany}
	\email{fgrube@math.uni-bielefeld.de}

	\makeatletter
	\@namedef{subjclassname@2020}{%
		\textup{2020} Mathematics Subject Classification}
	\makeatother
	
	
	
	\begin{abstract} 
		We provide a simple and direct proof of a strong-type unique continuation principle for the fractional $p$-Laplacian $(-\Delta_p)^s$ for a range of $s$ and $p$. The result extends to strong solutions of the fractional nonlinear Schr{\"o}dinger equation. We adapt the recent proofs of the weak UCP in \cite{BeSc26, Pra26}.
	\end{abstract}
	
	\hypersetup{pageanchor=false}
	\maketitle
	\hypersetup{pageanchor=true}
	
	\section{Introduction}\label{sec:intro}
	
	The unique continuation principle (UCP) is a fundamental property in the theory of partial differential equations. It states that if a solution to a partial differential equation, e.g., a Schrödinger equation $\Delta u +Vu=0$, vanishes in an open set (weak UCP) or to infinite order at some point (strong UCP), then it must already be zero. This principle is a fundamental tool in studying inverse problems or the stability of the Cauchy problem.
	
	For the $p$-Laplacian, the problem of unique continuation is open and only in spatial dimension $d=2$ is it partially resolved, see the contributions in the '80s \cite{Ale87, Man88, BoIw87}.
	
	For $s \in (0,1)$ and $p \in (1, \infty)$, the fractional $p$-Laplacian is defined up to a normalizing constant via the principal value integral
	\begin{equation}\label{eq:frac-p-lap}
		(-\Delta_p)^su(x)\coloneq \pv \int_{\R^d} \frac{|u(x)-u(y)|^{p-2}\big(u(x)-u(y)\big)}{|x-y|^{d+sp}}\d y.
	\end{equation}
	This operator arises naturally in various contexts, including non-Newtonian fluid mechanics, game theory, and anomalous diffusion, see, e.g., \cite{DPV12}. While existence and regularity theory for equations involving \eqref{eq:frac-p-lap} have seen immense progress over the last decade, unique continuation properties have remained an open problem due to the rigid incompatibility between traditional UCP techniques and the nonlinear nature of $(-\Delta_p)^s$. 
	
	In the recent preprint \cite{BeSc26}, the authors characterize the class of L{\'e}vy-operators for which the weak continuation principle holds. They provide a direct proof of the weak UCP which hinges on the nonlocality of the L{\'e}vy-operator. This idea was applied to nonlinear, nonlocal parabolic equations with rough potentials in \cite{Pra26}. 
	
	In this article, we generalize the approach in \cite{BeSc26, Pra26} in order to obtain a strong-type UCP. 	
	
	\begin{theorem}[Conditional strong-type UCP]\label{th:fn-sucp}
		Let $s\in (0,1)$, $p\in (1,\infty)$, $0\in \Omega\subset \R^d$ be an open set, let $f\in L^\infty(\Omega)$ be an inhomogeneity such that $|f(x)|=\mathcal{O}(|x|^{n})$ as $|x|\to 0$ for any $n\in \N$, and $V:\Omega\times \R \to \R$ be a potential such that $|x|^{n_0}|t|^{-\eps_0}V(x,t)\in L^\infty(\Omega)$ for some $\eps_0>0$ and some $n_0\in \N$. Assume that $u\in L^{p-1}_{sp}(\R^d)$ satisfies 
		\begin{enumerate}
			\item[(i)] $u\in C^{1,\alpha}_\loc(\Omega)$ for some $\alpha>1-(1-s)p$ and $p\ge 2$ or 
			\item[(ii)] $u\in C^{\beta}_\loc(\Omega)$ for some $sp'<\beta\le 1$ and $p>1$, or
			\item[(iii)] $u\in C^{\lfloor \beta\rfloor, \beta-\lfloor \beta \rfloor}_\loc(\Omega)$ for some $sp'<\beta< 2$ and $1<p\le 2$
		\end{enumerate}
		and is a strong solution to 
		\begin{equation}\label{eq:fn-schroedinger}
			(-\Delta_p)^su-V(\cdot, u)=f \text{ in }\Omega.
		\end{equation}
		We define the numbers $m_{i,k}\in [0,\infty]$ for all $i\in \{1,\dots, d\}$ and $k\in \N$
		\begin{equation*}
			m_{i,k}\coloneq \int_{\R^d} \frac{x_i^{2k}}{|x|^{4k}} \frac{|u(x)|^{p-1}}{|x|^{d+sp}}\d x.
		\end{equation*}
		If the following series diverges for all $i\in \{1,\dots, d\}$
		\begin{equation}\label{eq:carleman}
			\sum_{k=1}^{\infty}\frac{1}{m_{i,k}^{1/(2k)}}=+\infty,
		\end{equation} 
		then $u=0$ on $\R^d$.
	\end{theorem}
	
	Here, we use the conventions $1/0=+\infty$ and $1/\infty=0$.
	
	\begin{remark}\
		\begin{itemize}
			\item[(i)] Already in the linear case ($p=2$), in order to obtain a priori the regularity $C^{1,\alpha}$ for $\alpha>1-(1-s)p$ or $C^{\beta}$ for $\beta>sp'$ one would need to assume the potential and the inhomogeneity to have some regularity in a neighborhood of the origin. 
			\item[(ii)] The assumption \eqref{eq:carleman} resembles the Carleman condition in the moment problem, see \cite[Section 4.2]{Sch17}. It yields that $|u(x)|$ decays to arbitrary order at the origin. 
			\item[(iii)] Instead of \eqref{eq:carleman}, we could assume the stronger but more accessible assumption: there exists $M>0$ such that for all $k\in \N$
			\begin{equation*}
				\int_{B_{1/k}(0)} \frac{|u(x)|^{p-1}}{|x|^{d+sp+2k}}\d x\le (Mk)^{2k}.
			\end{equation*}
			\item[(iv)] Theorem \autoref{th:fn-sucp} easily extends to parabolic problems since the additional time derivative $\partial_t u$ vanishes to arbitrary order. 
			\item[(v)] In contrast to the local ($s=1$) counterpart, the UCP for fractional operators immediately yields $u=0$ on the full space $\R^d$. This is a consequence of the interactions in \eqref{eq:frac-p-lap} which incorporate values of $u$ on the full space. 
			\item[(vi)] \autoref{th:fn-sucp} is stronger than the weak UCP since the sequences $m_{i,k}$ are uniformly bounded if $u$ vanishes in a neighborhood of the origin and, thus, the Carleman-condition \eqref{eq:carleman} is easily satisfied.
		\end{itemize}
		
	\end{remark}
	
	The following corollary directly applies \autoref{th:fn-sucp} and yields a strong-type unique continuation for weak solutions to the fractional $p$-Laplace problem.
	\begin{corollary}
		Let $s\in (0,1)$, $p\in(1,\infty)$ such that $1/(1-s)\le p<2/(1-s)$, and $0\in \Omega\subset \R^d$ be an open set. Assume that $u\in W^{s,p}(\Omega)\cap L^{p-1}_{sp}(\R^d)$ is a weak solution to 
		\begin{equation*}
			(-\Delta_p)^su=0 \text{ in }\Omega.
		\end{equation*}
		If \eqref{eq:carleman} holds, then $u=0$ on $\R^d$. 
	\end{corollary}
	This corollary follows directly from \autoref{th:fn-sucp} coupled with the known interior regularity result \cite{BDLBS24, GJS25}.

	If one equips the fractional $p$-Laplacian, see \eqref{eq:frac-p-lap}, with an appropriate normalization constant $c_{d,p,s}\asymp s(1-s)$, then it converges to the classical $p$-Laplace operator. We refer to the discussion in \cite{dTGCV} for appropriate choices of normalization constants. We want to stress that the phenomenon we exploit in our proof is purely nonlocal and there is no real hope of proving the UCP for the $p$-Laplacian using this result asymptotically. 
	
	\begin{remark}
		It is not clear to me how to weaken the assumption \eqref{eq:carleman} to $\int_{B_R(0)}|u(y)|^{p-1}\d y = \mathcal{O}(R^n)$ as $R\to 0$ for all $n\in \N$ which is the typical assumption in the strong UCP, compare \cite[Example 4.22]{Sch17}.
	\end{remark}

	\subsection{Related literature} 
	
	The first contribution to the unique continuation principle for the fractional Laplacian was done in \cite{Rie38} using the Kelvin transform. In the past decades, there has been large interest in the UCP for nonlocal problems. A key ingredient in the proofs was the extension technique provided in the work \cite{CaSi07}. This allows to write certain nonlocal problems as traces of higher-dimensional local problems. The first article in which this method was applied in order to obtain the strong unique continuation principle was \cite{FaFe14}. Utilizing frequency function methods, they prove the strong unique continuation for the fractional Schrödinger equation with certain scaling critical Hardy potentials. Relying on Carleman estimates, the strong UCP for $(-\Delta)^su-Vu=0$ was generalized to a larger class of potentials in \cite{Rue15}. In \cite{Seo14, Seo15}, the weak UCP for open sets for the same equation was established for potentials $V\in L^{n/(2s),\infty}$. Generalizing the strategy by Garofalo and Lin to the fractional setting, in \cite{Yu17} unique continuation properties are studied. Quantitative UCP for operators like $\sum_i (-\partial_i^2)^s+q$ were established in \cite{GMR20}. In \cite{Rue19a}, quantitative uniqueness properties for perturbations of Riesz transforms were studied.
	
	The weak continuation principle for fractional Schrödinger equations is also proved using extensions techniques and the Carleman estimates from \cite{Rue15} in \cite{GSU20}.

	The UCP for the relativistic Schrödinger equation was studied in \cite{FaFe15}. Various qualitative and quantitative (global) unique continuation properties for the fractional discrete Laplacian are contained in the article \cite{FRR24}. In \cite{WMW24}, the decay properties of solutions to certain nonlocal nonlinear operators $L_p^su=0$ were studied using the Caffarelli-Silvestre extension. Higher-order fractional Laplace equations and the UCP was studied in \cite{FeFe20}. As a consequence of the analyticity in the related extension problem, in \cite{BaGa23, ArBa23, ArBa23b, ABDG23, BaGh25} the authors prove the weak UCP for fractional powers of parabolic equations. In \cite{Rue17}, the UCP for nonlocal operators on compact Riemannian manifolds is studied. A characterization of the UCP for a class of nonlocal pseudo-differential operators with weights in terms of some monotonicity of the eigenvalues was studied in \cite{FrIa20}.
	
	The unique continuation principle is a critical tool in the study of the Calder{\'o}n problem. For nonlocal operators inverse problems were studied in \cite{GSU20}, \cite{RuSa18}, \cite{RuSa20}, \cite{CLR20}, \cite{LLR20},  \cite{GRS20}, \cite{GLX17}, \cite{KRZ23}, \cite{KLZ24}, \cite{CRTZ24}, and \cite{RaZi24}.	
	
	\subsection{Strategy of the proof}
	
	The main observation in the proof of the weak UCP in \cite{BeSc26, Pra26} (applied to the fractional, nonlinear Schrödinger equation) is that the equality 
	\begin{equation*}
		(-\Delta_p)^su= Vu \text{ in }\Omega\ni 0
	\end{equation*}
	can be differentiated in the set $\Omega$ where $u$ vanishes. The explicit representation of the fractional $p$-Laplacian as an integro-differential operator yields, after differentiating under the integral, 
	\begin{equation}\label{eq:moments-strategy}
		\int_{\Omega^c} |u(y)|^{p-2}u(y) \partial_x^\alpha|x-y|^{-d-sp}\d y=0 \text{ for all }x\in \Omega.
	\end{equation}
	Then it remains to notice that, after an application of the Stone-Weierstraß theorem, the set of functions 
	\begin{equation*}
		\operatorname{span} \{y\mapsto \frac{y^\alpha}{|y|^{2|\alpha|}}\mid \alpha\in \N^d \}
	\end{equation*}
	is dense in $C_0(\Omega^c)$. \smallskip
	
	For the strong UCP, it is not possible to differentiate the equality $(-\Delta_p)^su= Vu$ in a neighborhood of the origin. Instead, we set up finite difference quotients centered at the origin and prove by hand that they converge. This is only possible since we assume our solution to be strong solutions. In contrast to the approach in \cite{BeSc26, Pra26}, we cannot invoke the Stone-Weierstraß theorem since the functions $y^\alpha/|y|^{2|\alpha|}$ do not belong to $C_0(\R^d)$. Instead, we use the Kelvin-transform on \eqref{eq:moments-strategy} (at $x=0$) in order to transform the problem into a $d$-dimensional Hamburger moment problem for the measures $J_p(u(z/|z|^2))_{\pm}|z|^{sp-d}\d z$. Uniqueness to the moment problem provided by the Carleman condition \eqref{eq:carleman} completes the proof.
	
	\subsection*{Acknowledgments} Thanks to Moritz Kassmann for comments on the manuscript. Financial support by the German Research Foundation (DFG - Project number 541771122) is gratefully acknowledged. 
	
	\section{Preliminaries}\label{sec:prelims}
	We briefly introduce the notation, conventions, boundary assumptions, and function spaces used throughout this article. 
	
	The positive part of a real number $r$ is denoted by $r_+\coloneq \max\{r,0\}$ and the negative part $r_-\coloneq (-r)_+$. The notation $\lfloor r\rfloor$ denotes the largest integer that is smaller than $r$. A ball with radius $r$ and with center $x\in \R^d$ is written as $B_r(x)$. At times, we omit the center point and simply write $B_r$. For a vector $x\in \R^d$, we use the convention $x=(x',x_d)$ where $x'\in \R^{d-1}$ and $x_d\in \R$. We write $e_d\coloneq (0,\dots, 0, 1)\in \R^d$. 
	
	In our proofs, we use generic constants $C$ which depend only on universal quantities. These constants may change from line to line. 
	
	The space of Hölder continuous functions on $\overline{\Omega}$ is written as $C^{\alpha}(\overline\Omega)$, its seminorm is denoted by $[\cdot]_{C^{\alpha}(\overline{\Omega})}$. We also use $\dot{C}^\alpha(\overline{\Omega})$ to denote the homogeneous set of continuous functions $f:\Omega\to \R$ such that $[f]_{C^{\alpha}(\Omega)}<\infty$.
	
	The finite difference quotient $ D_{\eps}^\alpha$ for some multiindex $\alpha\in \N$ and $\eps>0$ is defined inductively via $ D_{\eps}^0 u(x) \coloneq u(x)$ and for any $i\in \{1,\dots, d\}$ and any multiindex $\alpha $ we set
	\begin{equation*}
		 D_{\eps}^{e_i+\alpha}u(x)\coloneq \frac{ D_{\eps}^\alpha u(x+\eps e_i)-  D_{\eps}^\alpha u(x)}{\eps}.
	\end{equation*}
	
	The tail-space $L_{sp}^{p-1}(\R^d)$ is defined as the space of locally $(p-1)$-integrable functions such that the norm
	\begin{equation*}
		\|u\|_{L_{sp}^{p-1}(\R^d)}\coloneq \Big( \int_{\R^d} \frac{|u(x)|^{p-1}}{(1+|x|)^{d+sp}}\d x \Big)^{1/(p-1)}
	\end{equation*}
	is finite. By $W^{s,p}(\Omega)$, we denote the classical fractional Sobolev space (à la Gagliardo).
	
	We say that $u\in L_{sp}^{p-1}(\R^d)$ is a strong solution to \eqref{eq:fn-schroedinger} if the equality \eqref{eq:fn-schroedinger} holds for all $x\in \Omega$. Instead, we say that $u\in W^{s,p}(\Omega)\cap L_{sp}^{p-1}(\R^d)$ is a weak solution if for any test function $\phi\in C_c^\infty(\Omega)$
	\begin{equation*}
		\frac{1}{2}\int_{\R^d}\int_{\R^d} \frac{|u(x)-u(y)|^{p-2}(u(x)-u(y))(\phi(x)-\phi(y))}{|x-y|^{d+sp}}\d y \d x= \int_{\Omega} V(x,u(x))\phi(x)\d x.
	\end{equation*}
	In case of the general Schrödinger-potential, we assume $|x|^{-n_0}|u(x)|^{\eps_0}\in L^1_{\loc}(\Omega)$ to ensure the finiteness of the integral on the right-hand side.

	\section{Proof of {\autoref{th:fn-sucp}}}
	Let $R_0\in (0,1)$ such that $B_{R_0}(0)\subset\subset \Omega$. 
	
	Firstly, let us prove that the condition \eqref{eq:carleman} implies that $u(x)=\mathcal{O}(|x|^n)$ as $x\to 0$ for all $n\in \N$. We prove this via contradiction. Assume that $u(x)=\mathcal{O}(|x|^n)$ as $x\to 0$ does not hold for all $n\in \N$. Then we find a number $n_0$, $\eps>0$, and a sequence $x_l\in B_{R_0}(0)$ converging to 0 as $l\to \infty$ such that $|u(x_l)|\ge \eps |x_l|^{n_0}$. Since $u\in C^\beta(\overline{B_{R_0}(0)})$, we find a radius $r_l\coloneq C|x_l|^{n_0/\alpha}$ such that $u(x)\ge \eps |x_l|^{n_0}/2$ for all $x\in B_{r_l}(x_l)$. After taking a subsequence we can find $i\in \{1,\dots, d\}$ such that $|(x)_i|\ge C|x_l|$ for all $x\in B_{r_l}(x_l)$ and all $l\in \N$. This yields
	\begin{equation*}
		m_{i,k}\ge C \int_{B_{r_l}(x_l)} |x_l|^{(p-1)n_0-2k-d-sp}\d x \ge C |x_l|^{(p-1)n_0-2k-d-sp+nd/\alpha}.
	\end{equation*} 
	Since $x_l\to 0$ as $l\to \infty$, we find $k_0\in \N$ such that $m_{i,k}=+\infty$ for all $k\ge k_0$. Plugging this into \eqref{eq:carleman} yields $\sum_{k=1}^{\infty}\frac{1}{m_{i,k}^{1/(2k)}}<\infty$ which is a contradiction.

	We define $J_p(x)\coloneq |x|^{p-2}x$ and know that 
	\begin{equation*}
		U(x)\coloneq \pv\int_{\R^d}\frac{J_p(u(x)-u(y))}{|x-y|^{d+sp}}\d y
	\end{equation*}
	equals $f(x)+V(x,u(x))$ in a neighborhood of the origin. We prove the result in two steps. 
	
	\textit{Step 1.} In the first step, we prove that for any multiindex $\alpha\in \N^d$ the finite difference quotient of the function $U$
	\begin{equation}\label{eq:finite-difference-approx}
		 D_{\eps}^\alpha U(0)\to - \int_{\R^d}J_p(u(y)) (-1)^{|\alpha|}\partial^\alpha_y |y|^{-d-sp}\d y \text{ as }\eps \to 0.
	\end{equation}
	
	In order to prove this, we provide a few auxiliary convergence results which aid in the proof. 
	
	\textit{Claim A.} For any fixed $k\in \N$ and $x\in B_{k\eps}(0)$, the term 
	\begin{equation*}
		W_{1}\coloneq \frac{1}{2}\int_{B_{(k+1)\eps}(0)}\frac{J_p(u(x)-u(x+h))+J_p(u(x)-u(x-h))}{|h|^{d+sp}}\d h
	\end{equation*}
	converges to zero of arbitrary order $\eps^n$ for all $n\in \N$. 
	
	\textit{Proof of claim A.} This is the only step of the proof in which we need to differentiate which regularity assumption on $u$ holds true and whether $p\ge 2$ or not. If $u\in C_\loc^{1,\alpha}(\Omega)$ and $p\ge 2$, then, using the regularity of $J_p$, we estimate
	\begin{equation}\label{eq:w-1-1-estimate}
		\begin{aligned}
			|W_{1}|&\le C  \int_{B_{(k+1)\eps}(0)}\big(|u(x)-u(x +h)| + |u(x)-u(x -h)|\big)^{p-2}\\
			&\qquad \qquad \times\frac{|2u(x)-u(x+h)-u(x -h)|}{|h|^{d+sp}}\d h\\
			&\le \int_{B_{(k+1)\eps}(0)}|h|^{p-2} \frac{\eps^{\lambda n} |h|^{(1+\alpha)(1-\lambda)}}{|h|^{d+sp}}\d h\le \eps^{\lambda (n-1)+(1-s)p +\alpha(1-\lambda)-1}.
		\end{aligned}
	\end{equation}
	Here, $\lambda \in (0,1)$ is chosen such that $(1+\alpha)(1-\lambda)-2 +p(1-s)>0$ which is possible by the assumption on $\alpha$. 
	
	If instead $u\in C^{1,\alpha}_\loc(\Omega)$ and $p\le 2$, then we can use that $J_p$ is uniformly in $\dot{C}^{p-1}(\R)$:
	\begin{equation*}
		\begin{aligned}
			|W_{1}|&\le C  \int_{B_{(k+1)\eps}(0)}\frac{|2u(x)-u(x+h)-u(x -h)|^{p-1}}{|h|^{d+sp}}\d h\\
			&\le C \int_{B_{(k+1)\eps}(0)}\frac{\eps^{\lambda (p-1) n} |h|^{(p-1)(1+\alpha)(1-\lambda)}}{|h|^{d+sp}}\d h\le C \eps^{\lambda (p-1)(n-1) +(p-1)(1+\alpha)(1-\lambda)-sp}.
		\end{aligned}
	\end{equation*}

	This proves the asymptotics of $W_{1}$ in the case $u\in C^{1,\alpha}_\loc(\Omega)$. If instead $u\in C^{\beta}_\loc(\Omega)$, then 
	\begin{equation*}
		\begin{aligned}
			|W_{1}|&\le C  \int_{B_{(k+1)\eps}(0)}\|u\|_{L^\infty(B_{4\eps}(0))}^{(p-1)\lambda} |h|^{\beta(p-1)(1-\lambda)-d-sp}\d h\le C \eps^{n(p-1)\lambda} \eps^{(p-1)\beta(1-\lambda)-sp}.
		\end{aligned}
	\end{equation*}
	Here, $\lambda\in (0,1)$ is chosen such that $\beta(1-\lambda)-sp'>0$. This completes our analysis of the term $W_{1}$ and proves claim A.\smallskip
	
	\textit{Claim B.} For any $k\in \N$ and any $x\in B_{k\eps}(0)$, the term
	\begin{equation*}
		\begin{aligned}
			W_2(x)&\coloneq \int_{B_{(k+1)\eps}(x)^c}\frac{J_p(u(x)-u(y))}{|x-y|^{d+sp}}\d y + \int_{B_{(k+1)\eps}(0)^c} \frac{J_p(u(y))}{|y|^{d+sp}}\d y\\
			&\qquad+  \int_{B_{(k+1)\eps}(0)^c}J_p(u(y))\Big(\frac{1}{|x-y|^{d+sp}} -\frac{1}{|y|^{d+sp}}\Big)\d y
		\end{aligned}
	\end{equation*}
	converges to zero of arbitrary order. 
	
	\textit{Proof of Claim B.} We split the term in two further terms $W_2= W_{2,1}+ W_{2,2}$ where
	\begin{equation*}
		W_{2,1}(x)\coloneq \int_{B_{(k+1)\eps}(0)^c}\frac{J_p(u(x)-u(y))+J_p(u(y))}{|x-y|^{d+sp}}\d y
	\end{equation*}
	and 
	\begin{equation*}
		W_{2,2}(x)\coloneq \int_{B_{(k+1)\eps}(x)^c}\frac{J_p(u(x)-u(y))}{|x-y|^{d+sp}}\d y-\int_{B_{(k+1)\eps}(0)^c}\frac{J_p(u(x)-u(y))}{|x-y|^{d+sp}}\d y.
	\end{equation*}
	
	In order to treat $W_{2,1}$, we use a $C^{\lambda}$-estimate for $J_p$ for some $\lambda\in (0,1)$ such that $p-1-\lambda>0$. This yields
	\begin{equation*}
		\begin{aligned}
			|W_{2,1}|&\le C  \int_{B_{(k+1)\eps}(0)^c}\frac{\big( |u(x)-u(y)| + |u(y)| \big)^{p-1-\lambda} |u(x)|^{\lambda}}{|x-y|^{d+sp}}\d y\\
			&\le C  \eps^{\lambda n}\int_{B_1(0)\setminus B_{(k+1)\eps}(0)} |y|^{n(p-1-\lambda)-d-sp} \d y + C \eps^{\lambda n}\|u\|_{L_{sp}^{p-1}(\R^d)  }^{p-1-\lambda}.
		\end{aligned}
	\end{equation*}
	Now, the term $W_{2,2}$ becomes after manipulating the integration domains
	\begin{equation*}
			W_{2,2}= \int_{B_{(k+1)\eps}(x)^c\cap B_{(k+1)\eps}(0)}\frac{J_p(u(x)-u(y))}{|x-y|^{d+sp}}\d y-\int_{B_{(k+1)\eps}(0)^c\cap B_{(k+1)\eps}(x)}\frac{J_p(u(x)-u(y))}{|x-y|^{d+sp}}\d y.
	\end{equation*}
	This allows us to estimate both terms, using, again, $u(y)=\mathcal{O}(|y|^{n})$ as $|y|\to 0$
	\begin{equation*}
		|W_{2,2}|\le C  \Big( \int_{B_{(k+1)\eps}(x)^c\cap B_{(k+1)\eps}(0)}+\int_{B_{(k+1)\eps}(0)^c\cap B_{(k+1)\eps}(x)}\Big)
		(|y|+\eps)^{(p-1)n-d-sp}\d y\le C  \eps^{(p-1)n-sp}.
	\end{equation*}
	This proves the claim B.\smallskip
	
	These two claims allow us to complete the first step of this proof. If $|\alpha|=0$, then the proof is trivial by definition. We define two auxiliary function $U_{1}$ and $U_{2}$ via
	\begin{equation*}
		\begin{aligned}
			U_{1}(x)\coloneq \frac{1}{2}\int_{B_{(k+1)\eps}(0)}\frac{J_p(u(x)-u(x+h))+J_p(u(x)-u(x-h))}{|h|^{d+sp}}\d h,\\
			U_{2}(x)\coloneq \frac{1}{2}\int_{B_{(k+1)\eps}(0)^c}\frac{J_p(u(x)-u(x+h))+J_p(u(x)-u(x-h))}{|h|^{d+sp}}\d h.
		\end{aligned}
	\end{equation*}
	Clearly, $U(x)=U_1(x)+U_2(x)$. Due to claim $A$, any finite difference quotient of $U_1$ vanishes, i.e., $ D_{\eps}^\alpha U_1(0)\to 0$ as $\eps \to 0$.
	
	Let $i\in \N$ such that $\alpha = e_i + \beta$ for some multiindex $\beta\in \N^d$. Let us consider a single difference of $U_2$. We may write 
	\begin{equation*}
		\begin{aligned}
			\frac{U_2(\eps e_i)-U_2(0)}{\eps}= \frac{W_2(\eps e_i)}{\eps}- \int_{B_{(k+1)\eps}(0)^c}J_p(u(y))\frac{|\eps e_i-y|^{-d-sp}-|y|^{-d-sp}}{\eps}\d y.
		\end{aligned}
	\end{equation*}
	Due to claim B, the first term, i.e., $W_{2}(\eps e_i)/\eps$ converges to zero of arbitrary order. This immediately yields that in the finite difference $ D_{\eps}^\alpha U_2(0)$ the only term that survives the limit $\eps \to 0$ is 
	\begin{equation*}
		\int_{B_{(k+1)\eps}(0)^c}J_p(u(y))  D_{\eps}^\alpha|(\cdot)-y|^{-d-sp}\d y.
	\end{equation*}
	Since $J_p(u(y))= \mathcal{O}(|y|^{(p-1)n})$ as $|y|\to 0$ for any $n\in \N$ and $J_p(u(y))/|y|^{d+sp}\in L^1(B_1(0)^c)$, due to dominated convergence, this term converges to 
	\begin{equation*}
		\int_{\R^d}J_p(u(y)) (-1)^{|\alpha|}\partial^\alpha_y |y|^{-d-sp}\d y \text{ as }\eps \to 0.
	\end{equation*}\smallskip
	
	\textit{Step 2.} In this step, we consider the derivative of the potential. Due to the fact that both $u$ and $f$ are vanishing of any order at the origin and using the assumption on the potential $V$, it is clear that
	\begin{equation*}
		\begin{aligned}
			| D_{\eps}^\alpha \big(f+V(\cdot,u(\cdot))\big)(0)|&\le C\eps^{n-|\alpha|}+  C \sup_{x\in B_{|\alpha|\eps}(0)} \frac{|V(x,u(x))|}{\eps^{|\alpha|}}\le C\eps^{n-|\alpha|}+  C \sup_{x\in B_{|\alpha|\eps}(0)} \frac{|u(x)|^{\eps_0}}{|x|^{n_0}\eps^{|\alpha|}}\\
			&\le C\eps^{n-|\alpha|}+ C \eps^{n\eps_0-n_0-|\alpha|}\to 0
		\end{aligned}
	\end{equation*}
	as $\eps\to0$ if we pick $n$ sufficiently large.\smallskip
		
	\textit{Step 3.} From step 1 and step 2, using an induction as in \cite{BeSc26}, we know that 
	\begin{equation*}
		\int_{\R^d} \frac{J_p(u(y))}{|y|^{d+sp}} \frac{y^\alpha}{|y|^{2|\alpha|}}\d y =0
	\end{equation*}
	for any multiindex $\alpha\in \N^d$. Applying the Kelvin transform ($z=y/|y|^2$) yields 
	\begin{equation}\label{eq:moment-problem}
		\int_{\R^d} z^\alpha \tilde{U}(z)\d z =0 
	\end{equation}
	for any multiindex $\alpha$. Here, $\tilde{U}(z)\coloneq J_p(u(z/|z|^2))|z|^{sp-d}$. From \eqref{eq:moment-problem}, we immediately get 
	\begin{equation*}
		\int_{\R^d} z^\alpha \tilde{U}_+(z)\d z= \int_{\R^d} z^\alpha \tilde{U}_-(z)\d z.
	\end{equation*}	
	Due to \cite{Nus65}, \cite[Theorem 3]{Pet82} and \cite[Theorem 4.3]{Sch17}, this moment problem yields $\tilde{U}_+=\tilde{U}_-$ since the Carleman condition \eqref{eq:carleman} is satisfied. From this the equality $u=0$ follows immediately.\qed
	
	
	\newcommand{\etalchar}[1]{$^{#1}$}

\end{document}